\documentclass[a4paper]{article}
\usepackage{amsmath}
\usepackage{amsthm}
\usepackage[all]{xy}
\usepackage[english]{babel}
\usepackage[utf8]{inputenc}
\usepackage{graphicx}
\usepackage{xypic,color}
\usepackage{amsfonts, amssymb}

\newtheorem{thm}{Theorem}
\newtheorem{defn}[thm]{\textbf{Definition}}

\newtheorem{cor}[thm]{Corollary}
\newtheorem{lem}[thm]{Lemma}
\newtheorem{prop}[thm]{Proposition}
\theoremstyle{definition}

\theoremstyle{definition}

\theoremstyle{remark}
\newtheorem{rem}[thm]{Remark}

\renewcommand{\k}{\Bbbk}

\renewcommand{\d}{\Delta}

\newcommand{\al}{\alpha}
\newcommand{\be}{\beta}

\newcommand{\ot}{\otimes}

\newcommand{\rt}{\rightarrow}

\title{Nearly Frobenius structures in some families of algebras}
\author{Dalia Artenstein, Ana Gonz\'alez, Gustavo Mata.
}

\date{\today}
\begin{document}
\maketitle

\begin{abstract}
In this article we continue with the study started in \cite{AGL15} of nearly Frobenius structures in some representative families of finite dimensional algebras, as radical square zero algebras, string algebras and toupie algebras. We prove that such radical square zero algebras with at least one path of length two are nearly Frobenius. As for the string algebras, those who are not gentle can be endowed with at least one non-trivial nearly Frobenius structure. Finally, in the case of toupie algebras, we prove that the existence of monomial relations is a sufficient condition to have non-trivial nearly Frobenius structures. Using the technics developed for the previous families of algebras we prove sufficient conditions for the existence of non-trivial Frobenius structures in quotients of path algebras in general.
\end{abstract}



\section{Introduction}
\label{intro}
 
It is a well known result that \emph{the Poincar\'e algebra} $A=H^\ast(M)$ associated to a compact closed manifold $M$  with trace
$$\varepsilon(w)=\int_Mw,$$
for $w\in H^\ast(M)$ is a Frobenius algebra. This is not the case for a non-compact manifold $M$, but we may ask what structure remains. The answer may be stated nowaday as follows: the cohomology algebra is a nearly Frobenius algebra.

The concept of nearly Frobenius algebra was developed in the thesis of the second author of this article, the study of these objects was motivated by the result proved in \cite{CG04}, which states that: \emph{the homology of the free loop space $H^*(LM)$ has the structure of a Frobenius algebra without counit}. Later, these objets were studied in \cite{GLSU} and their algebraic properties were developed in \cite{AGL15}. In particular, the possible nearly Frobenius structures in gentle algebras were described.

In the framework of differential graded algebras, Abbaspour considers  in \cite{A}  nearly Frobenius algebras that he calls open Frobenius algebras. He proves that if $\mathcal{A}$ is a symmetric open Frobenius algebra of degree $m$, then $HH_{\ast}(\mathcal{A},\mathcal{A})[m]$ is an open Frobenius algebra. 

The Frobenius algebra structures in quotients of path algebras have been deeply studied. In particular, there is only a small family of monomial algebras that admits Frobenius structure. The next result illustrates this assertion.
\begin{lem}[\cite{CHYZ}, Lemma 2.2]
Let $\mathcal{A}$ be an indecomposable monomial algebra. Then $\mathcal{A}$ is Frobenius if and only if $\mathcal{A} = \k$, or $\mathcal{A}\cong \frac{\k Z_n}{J^d}$ for some positive integers $n$ and $d$, with $d\geq 2$, where  $Z_n$ the basic oriented cycle of length $n$ and $J$ is the ideal generated by the arrows.
\end{lem}

If we are interested in finding nearly Frobenius structures the situation is quite different. A large number of families of finite dimensional algebras can be endowed with nearly Frobenius structure. In particular, in this work, we study the nearly Frobenius structures in radical square zero algebras, string algebras and toupie algebras.

The radical square zero algebras are finite dimensional algebras over a field $\k$ such that the square of its Jacobson radical is already zero. They have been extensively studied in representation theory because their behaviour provides interesting examples and results. Some beautiful results about them are the following:
a radical square zero algebra is of representation finite type if and only if its separated quiver is a finite disjoint union of Dynkin diagrams (see \cite[Chapter X2]{ARS}).
A connected radical square zero algebra is either self-injective or CM-free (see \cite{C}). As a consequence of the last result, we have that a connected radical square zero algebra is Gorenstein if and only if its valued quiver is either an oriented cycle with the trivial valuation or does not contain oriented cycles.

String algebras are on one hand special biserial algebras whose ideal of relations can be generated by paths and on the other hand a generalization of gentle algebras.
The class of special biserial algebras was introduced by Skowro\'nski and Waschbüsch in \cite{SW}. It has played an important role in the study of self-injective algebras. Special biserial algebras and in particular string algebras have a well-understood representation theory. In fact, if $A$ is special biserial, then it has a two-sided ideal $S$ such that the quotient $A/S$ is a monomial algebra, and actually a string algebra.

The other family of algebras studied in this article are the toupie algebras. They were first introduced in ``Toupie algebras, some examples of Laura algebra'' by Diane Castonguay, Julie Dionne, Francois Huard and Marcelo Lanzilotta (see \cite{CDHL} for more details).
These algebras combine features of canonical algebras with monomial algebras. They are
quotients of the path algebra of a finite quiver $Q$ which has a source $0$, a sink $\omega$ and branches
going from $0$ to $\omega$. The ideal of relations $\mathcal{I}\subset Q_{\leq2}$ can be generated by a set containing two types
of relations: monomial ones, which involve arrows of one branch each, and linear combinations
of branches.\\

This work is developed as follows. In section 2 is devoted to study radical square zero algebras, in particular we determine the Frobenius dimension for all them. In Section 3 we study the string algebras and we prove that if $A $ is a not gentle string algebra then it admits at least one non-trivial nearly Frobenius coproduct. To finish the study of the algebras mentioned above we describe, in Section 4, the toupie algebras and the conditions ensuring that a toupie algebra admits nearly Frobenius coproducts. 
In the last section, based on all the previous results, we give conditions on quotients of path algebras so that they admit nearly Frobenius structures.


\section{Preliminaries}

Throughout this article $\Bbbk$ always denote a field. We denote by $\mathcal{A}$ a finite dimensional $\Bbbk$-algebra, and by Proposition 6 of $\cite{AGM}$, $\mathcal{A}$ can be considered as a connected algebra.

\begin{defn}
A \emph{quiver} $Q = \bigl(Q_0,Q_1, s, t\bigr)$ is a quadruple consisting of two sets: $Q_0$ (whose elements are called \emph{points}, or \emph{vertices}) and $Q_1$ (whose elements are called \emph{arrows}), and two maps $s, t : Q_1 \rt Q_0$ which associate to each arrow $\al\in Q_1$ its \emph{source} $s(\al)\in Q_0$ and its \emph{target} $t(\al)\in Q_0$, respectively.
\end{defn}
An arrow $\al\in Q_1$ of source $a = s(\al)$ and target $b = t(\al)$ is usually denoted by $\al: a\rt b$. A quiver $Q = \bigl(Q_0,Q_1, s,t\bigr)$ is usually denoted simply by $Q$. Thus, a quiver is nothing but an oriented graph without any restriction on the number of arrows between two points, the existence of loops or oriented cycles.

\begin{defn}
Let $Q = \bigl(Q_0,Q_1, s, t\bigr)$ be a quiver and $a, b\in Q_0$. A \emph{path} of \emph{length} $l\geq 1$ with source $a$ and target $b$ (or, more briefly, from $a$ to $b$) is a sequence
$$\bigl(a |\al_1,\al_2,\dots,\al_l| b\bigr),$$
where $\al_k\in Q_1$ for all $1\leq k\leq l$, $s\bigl(\al_1\bigr)=a$, $t\bigl(\al_k\bigr)=s\bigl(\al_{k+1}\bigr)$ for each $1\leq k<l$, and $t\bigl(\al_l\bigr)=b$. Such a path is
denoted briefly by $\al_1\al_2\dots\al_l$.
\end{defn}

\begin{defn}
Let $Q$ be a quiver. The \emph{path algebra} $\k Q$ is the $\k$-algebra whose underlying $\k$-vector space has as its basis the set of all paths $\bigl(a|\al_1,\al_2,\dots,\al_l|b\bigr)$ of length
$l\geq 0$ in $Q$ and such that the product of two basis vectors\\
$\bigl(a|\al_1,\al_2,\dots,\al_l|b\bigr)$ and $\bigl(c|\be_1,\be_2,\dots,\be_k|d\bigr)$ of $\k Q$ is defined by
$$\bigl(a|\al_1,\al_2,\dots,\al_l|b\bigr)\bigl(c|\be_1,\be_2,\dots,\be_k|d\bigr)=\delta_{bc}\bigl(a|\al_1,\dots,\al_l,\be_1,\dots,\be_k|d),$$
where $\delta_{bc}$ denotes the Kronecker delta. In other words, the product of two paths $\al_1\dots\al_l$ and $\be_1\dots\be_k$ is equal to zero if $t\bigl(\al_l\bigr) \neq s\bigl(\be_1\bigr)$ and is
equal to the composed path $\al_1\dots\al_l\be_1\dots\be_k$ if $t\bigl(\al_l\bigr) = s\bigl(\be_1\bigr)$. The product of basis elements is, then, extended to arbitrary elements of $\k Q$ by 
distributivity.
\end{defn}

Let $Q$ be a quiver and  $\k Q$ be the associated path algebra. Denote by $R_Q$ the two-sided ideal in $\k Q$ generated by all paths of length 1, i.e. all arrows. This ideal is known as the \emph{arrow ideal}.\\
It is easy to see that, for any $m\geq 1$ we have that $R^m_Q$ is a two-sided ideal generated by all paths of length $m$. Note that we have the following chain of ideals:
$$R^2_Q\supseteq R^3_Q \supseteq R^4_Q\supseteq\cdots$$

\begin{defn}
A two-sided ideal $\mathcal{I}$ in $\k Q$ is said to be \emph{admissible} if there exists $m\geq 2$ such that $$R^m_Q\subseteq \mathcal{I}\subseteq R^2_Q.$$

Also, the algebra $\frac{\Bbbk Q}{\mathcal{I}}$ is \emph{monomial} if $\mathcal{I}$ is generated by paths.
\end{defn}

An algebra $\frac{\Bbbk Q}{\mathcal{I}}$ is a connected algebra if and only if $Q$ is a connected quiver (See Lema 1.7, chapter II of \cite{ASS}), so, from now on, we only consider connected quivers. 

\begin{defn}
	An algebra $\mathcal{A}$ is a \emph{nearly Frobenius algebra} if it admits a linear map $\Delta:\mathcal{A}\rt \mathcal{A}\ot \mathcal{A}$ such that
	$$\xymatrix{\mathcal{A}\otimes \mathcal{A}\ar[r]^{m}\ar[d]_{\Delta\otimes 1}& \mathcal{A}\ar[d]^{\Delta}\\
		\mathcal{A}\otimes \mathcal{A}\otimes \mathcal{A}\ar[r]_{1\otimes m}&\mathcal{A}\otimes \mathcal{A}
	},\quad \xymatrix{\mathcal{A}\otimes \mathcal{A}\ar[r]^{m}\ar[d]_{1\otimes\Delta}& \mathcal{A}\ar[d]^{\Delta}\\
		\mathcal{A}\otimes \mathcal{A}\otimes \mathcal{A}\ar[r]_{m\otimes 1}&\mathcal{A}\otimes \mathcal{A}
	}$$ commute.
\end{defn}

In \cite{AGM} was showed that the previous definition agree with the definition given in \cite{AGL15}.

\begin{defn}
	The \emph{Frobenius space} associated to an algebra $\mathcal{A}$ is the vector space of all the possible coproducts $\d$ that make it into a nearly Frobenius algebra ($\mathcal{E}$), see \cite{AGL15}. Its dimension over $\k$ is called the \emph{Frobenius dimension} of $\mathcal{A}$, that is,
	$$\operatorname{Frobdim}\mathcal{A} = \operatorname{dim}_\k\mathcal{E}.$$
\end{defn}

Gabriel's Theorem states that if $\mathcal{A}$ is a basic and connected finite dimensional $\k$-algebra over a algebraically closed field $\k$, there exist a quiver $Q$ and an ideal $\mathcal{I}$ of $\k Q$ such that $\mathcal{A} \cong \frac{\k Q}{\mathcal{I}}$. This motivates us to study the existence of nearly Frobenius structures on quotients of path algebras.  


\section{Radical square zero algebras}

\begin{defn}
A \textbf{radical square zero algebra} is a finite dimensional algebra over al field $\k$ such that the square of its Jacobson radical is already zero.
\end{defn}

Now, if we consider a radical square zero algebra $\mathcal{A}$ associated to $Q$, we can determine the general expression of a nearly Frobenius coproduct over a general vertex of $Q$ and, using that the composition of two arrows is zero, determine the Frobenius dimension of the algebra $\mathcal{A}$.

\begin{rem}\label{lemita}
Let $\mathcal{A}=\frac{\k Q}{\mathcal{I}}$ be a radical square zero algebra and $p\in Q_0$. We can describe all the possible coproducts over $p$, depending if $p$ is sink, source, or an intermediate vertex.

\end{rem}
First, remember that if $\d$ is a nearly Frobenius coproduct then $\d\bigl(e_p\bigr)=\bigl(e_p\otimes 1\bigr)\d\bigl(e_p\bigr)=\d\bigl(e_p\bigr)\bigl(1\otimes e_p\bigr)$, for $ p\in Q_{0}$.

\begin{itemize}
  \item If $p$ is a source we can distinguish two different situations depending on the outdegree of $p$. Let $n=gr^{+}(p)$ be the outdegree of $p$.\\
  If $n=1$ the situation is the following
  \begin{center}
  $\xymatrix{p\ar[r]^\beta&q_1\cdots}$,
  \end{center}

  then $\d\bigl(e_p\bigr)=a_0e_p\ot e_p+a_1\beta\ot e_p$ and $\d(\beta)=\d\bigl(e_p\bigr)(1\otimes\be)=a_0e_p\ot \beta +a_1\beta\ot \beta=(\be\ot 1)\d\bigl(e_{q_1}\bigr)$. We deduce that $a_0=0$ and obtain that
  $$\d\bigl(e_p\bigr)=a_1\be\ot e_p.$$
 If $n\geq 2$ we have the following subquiver $$
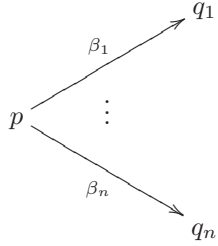
\xymatrix{&&q_1\\ p\ar[rrd]_{\beta_n}\ar[rru]^{\beta_1}&\vdots\\&&q_n}$$\\ $\d\bigl(e_p\bigr)=a_0e_p\ot e_p+\sum_{i=1}^nb_i\be_i\ot e_p$, then $\d\bigl(\be_j\bigr)=\d\bigl(e_p\bigr)\bigl(1\ot\be_j\bigr)=a_0e_p\ot \be_j+\sum_{i=1}^nb_i\be_i\ot \be_j=\bigl(\be_j\ot 1\bigr)\d\bigl(e_{q_j}\bigr)$. Therefore $a_0=0$ and $b_i=0$ for $i\neq j$. Using that $n\geq 2$ we conclude that $b_i=0$ for all $i=1,\dots, n$. Then
 $$\d\bigl(e_p\bigr)=0.$$
 \item If $p$ is a sink, similarly to the previous case, calling $m$ the indegree of $p$, we can conclude that:
if $m=1$ and $\al$ is the arrow ending in $p$, $\d\bigl(e_p\bigr)=a_1e_p\ot \al.$\\


In the case that $m\geq 2$ we have the following situation $$
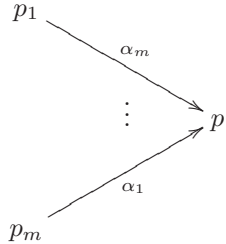
\xymatrix{p_1\ar[rrd]^{\al_m}&&\\ &\vdots& p\\p_m\ar[rru]_{\al_1}&&}$$
Reproducing the argument used in the case where $p$ is a source, we conclude that $\d\bigl(e_p\bigr)=0$.
  \item The last case is when $p$ is an intermediate vertex: Let $m$ be the indegree and $n$ the outdegree of $p$. If $m\geq 2$ and $n\geq 2$ we have the following representation
  $$
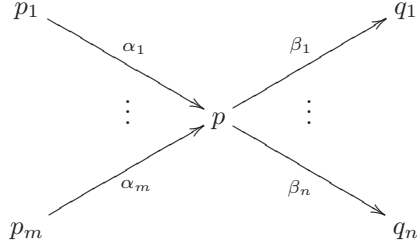
\xymatrix{
  p_1\ar[rrd]^{\al_1}&&&&q_1\\
  &\vdots& p\ar[rru]^{\be_1}\ar[rrd]_{\be_n}&\vdots\\
  p_m\ar[rru]_{\al_m}&&&&q_n
  }$$
  $$\d\bigl(e_p\bigr)=\bigl(e_p\ot 1\bigr)\d\bigl(e_p\bigr)=\d\bigl(e_p\bigr)\bigl(1\ot e_p\bigr)$$
  $$=a_0e_p\ot e_p+\sum_{i=1}^ma_ie_p\ot\al_i+\sum_{j=1}^nb_j\be_j\ot e_p+\sum_{i,j=1}^{m,n}c_{ij}\be_j\ot\al_i.$$
  $\displaystyle{\d\bigl(\al_k\bigr)=\bigl(\al_k\ot 1\bigr)\d\bigl(e_p\bigr)= a_0\al_k\ot e_p+\sum_{i=1}^ma_i\al_k\ot\al_i=\d\bigl(e_{p_k}\bigr)\bigl(1\ot \al_k\bigr)},$ \\then $a_0=0$, and $a_i=0$ for all $i\neq k$, for $k=1,\dots, m$. Similarly\\
  $\displaystyle{\d\bigl(\be_k\bigr)=\d\bigl(e_p\bigr)\bigl(1\ot\be_k\bigr)= \sum_{j=1}^nb_j\be_j\ot\be_k=\bigl(\be_k\ot 1\bigr)\d\bigl(e_{q_k}\bigr)},$ then $b_j=0$ for all $j\neq k$, for $k=1,\dots, n$. Then
  $$\d\bigl(e_p\bigr)=\sum_{i,j=1}^{m,n}c_{ij}\be_j\ot\al_i.$$
  If one or more arrows are loops, then the result is equal to the previous case.\\
  If $m=1$ and $n=1$ and it is not an isolated loop, then $\displaystyle{\d\bigl(e_p\bigr)=ae_p\ot\al+b\be\ot e_p+c\be\ot\al}$.\\
  If $m=1$ and $n=1$ with loop $\al$ in $e_p$, then $\displaystyle{\d\bigl(e_p\bigr)=a(e_p\ot\al+\al\ot e_p)+b\al\ot\al}$.\\
  If $m=1$ and $n\geq 2$ with local representation of the quiver $$
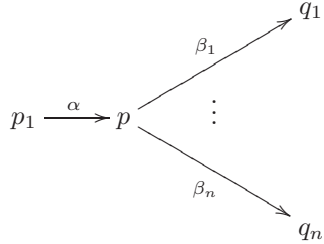
\xymatrix{
  &&&q_1\\
  p_1\ar[r]^{\al}& p\ar[rru]^{\be_1}\ar[rrd]_{\be_n}&\vdots&\\
  &&&q_n
  }$$
  $\displaystyle{\d\bigl(e_p\bigr)=a_0e_p\ot e_p+a_1e_p\ot\al+\sum_{j=1}^nb_j\be_j\ot e_p+\sum_{j=1}^{n}c_{j}\be_j\ot\al}$, then $\d(\al)=a_0\al\ot e_p+a_1\al\ot\al=\d\bigl(e_p\bigr)(1\ot\al)$, therefore $a_0=0$. Similarly, \\
  $\displaystyle{\d\bigl(\be_k\bigr)=\d\bigl(e_p\bigr)\bigl(1\ot \be_k\bigr)=\sum_{j=1}^nb_j\be_j\ot\be_k=(\be\ot 1)\d\bigl(e_{q_1}\bigr)}$, then $b_j=0$ for al $j\neq k$, for $k=1,\dots, n$. Finally, $$\d\bigl(e_p\bigr)=a_1e_p\ot\al+\sum_{j=1}^{n}c_{j}\be_j\ot\al.$$
  If $m=1$ and $n\geq 2$ with a loop $\al$ in $e_p$ and $\be_1 \ldots \be_{n-1}$ arrows starting in $e_p$, then $$\d\bigl(e_p\bigr)=a_1\al \ot\al+\sum_{j=1}^{n-1}c_{j}\be_j\ot\al.$$
  If $m\geq 2$ and $n=1$ and the quiver is locally the following $$
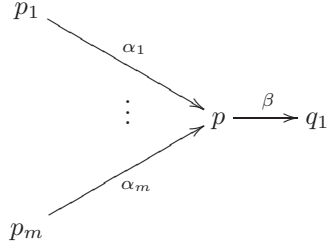
\xymatrix{
  p_1\ar[rrd]^{\al_1}&&&\\
  &\vdots& p\ar[r]^{\be}&q_1\\
 p_m\ar[rru]_{\al_m} &&&
  }$$
  As before $\displaystyle{\d\bigl(e_p\bigr)=b_1\be\ot e_p+\sum_{i=1}^mc_i\be\ot\al_i}.$\\
  Finally if  $m\geq 2$ and $n=1$ with a loop $\be$ in $e_p$ and $\al_1 \ldots \al_{m-1}$ arrows ending in $e_p$, then $$\displaystyle{\d\bigl(e_p\bigr)=b_1\be\ot \be+\sum_{i=1}^{m-1}c_i\be\ot\al_i}.$$
\end{itemize}

The previous remark allows us to construct all the nearly Frobenius coproducts that the radical square zero algebra admits.

\begin{cor}
If $\mathcal{A}=\frac{\k Q}{\mathcal{I}}$ with $\mathcal{I}\neq 0$ is a radical square zero algebra then $\operatorname{Frobdim}\mathcal{A}>0$.
\end{cor}
\begin{proof}
Since $\mathcal{I}\neq 0$ there is at least one path of length two. Let us call $p$ the intermediate vertex of that path, $m$ the indegree and $n$ the outdegree of $p$. Then the situation is similar to the last case of the previous remark and $\operatorname{Frobdim}\mathcal{A}\geq m.n\geq 1$.

\end{proof}






Finally we determine the Frobenius dimension of radical square zero algebras. For this reason we introduce the following useful notations for some special subspaces of $\Bbbk Q$ and $\Bbbk Q \otimes \Bbbk Q$: $V_p = \langle \{\alpha\}_{s(\alpha) = p, \ \alpha \in Q_1} \rangle $, $W_p = \langle \{\beta\}_{t(\beta) = p, \ \alpha \in Q_1} \rangle$, $\bar{V}_p = \langle V_p, e_p \rangle$, $\bar{W}_p = \langle W_p, e_p \rangle$, where $p \in Q_0$, and $U_{\alpha} = \langle \{\alpha \otimes e_p - e_q \otimes \alpha\}_{s(\alpha) = p, t(\alpha) = q } \rangle$ where $\alpha \in Q_1$. From Remark \ref{lemita} we can obtain the following result:

\begin{prop}\label{P15}
If $\mathcal{A}=\frac{\k Q}{\mathcal{I}}$ is a radical square zero algebra, then

\small{$$\operatorname{Frobdim}\mathcal{A} = \operatorname{dim}_\k \frac{\sum_{p \in Q_0} V_p \otimes \bar{W}_p + \bar{V}_p\otimes W_p}{ \sum_{p \in Q_0, gr^{+} (p) \geq 2} V_p \otimes \langle e_p \rangle + \sum_{p \in Q_0, gr^{-} (p) \geq 2 }  \langle e_p \rangle \otimes W_p + \sum_{\alpha \in Q_1} U_{\alpha} }$$}

\end{prop}

\begin{proof}
Suppose that $\d$ is a nearly Frobenius coproduct on $\mathcal{A}$. It is clear that $\d(e_p) \in V_p \otimes \bar{W}_p + \bar{V}_p\otimes W_p + \langle e_p\otimes e_p \rangle $.

\underline{Claim}: if $\d (e_p) = ae_p\otimes e_p + v+w$ with $v \in V_p \otimes \bar{W}_p$ and $w \in \bar{V}_p\otimes W_p$, then $a = 0$.

Let $\alpha$ be an arrow of $Q_1$ such that $s(\alpha) = p$ and $t(\alpha) = q$. Since $\d(e_p) = ae_p\otimes e_p + v + w$, then $\d(\alpha) = a(e_p \otimes \alpha) + v' = b\alpha \otimes e_q + w'$ with $v' \in V_p \otimes W_p$ and $w' \in  V_q \otimes W_q$. It is easy to see that $\alpha \otimes e_p \not\in V_q\otimes W_q + \langle e_q \otimes \alpha \rangle$, hence $a = 0$.

Then $\d (e_p)$ can be generated by vectors in $V_p \otimes \bar{W}_p + \bar{V}_p\otimes W_p$ for all $p \in Q_0$.

\underline{Claim}: if $ \gamma_0$ is an arrow of $Q_1$ such that $s(\gamma_0) = p$ and $t(\gamma_0) = q$,
$\d(e_p) = v + \sum c_{\alpha} \alpha \otimes e_p$ where $v \in \bar{V}_p \otimes W_p$ and $\d(e_q) = w + \sum d_{\beta} e_q \otimes \beta$ where $w \in V_q \otimes \bar{W}_q$ then $c_{\gamma_0} = d_{\gamma_0}$.

Using that $\d$ is a nearly Frobenius structure we have that $\d(\gamma_0) = (\gamma_0 \otimes 1) \d(e_q) = \d(e_p)(1\otimes \gamma_0)$.

\begin{itemize}
\item $(\gamma_0 \otimes 1) \d(e_q) = \sum d_{\beta} \gamma_0 \otimes \beta$

\item $\d(e_p)(1\otimes \gamma_0) = \sum c_{\alpha} \alpha \otimes \gamma_0$

\end{itemize}

On the other hand the set $\{\gamma_0 \otimes \beta \}_{\beta \in B} \cup \{ \alpha \otimes \gamma_0\}_{\alpha \in A}$ is linearly independents, then the equations agree only if $c_{\gamma_0} = d_{\gamma_0}$.

Finally, using Remark \ref{lemita} when the indegree or outdegree are greater or equal to two, it follows the thesis statement. \end{proof}


\section{String algebras}

The class of special biserial algebras was studied by Skowronski and Waschb\"{u}sch in \cite{SW} where they characterize the biserial algebras of finite representation type. The definition
of these algebras can be given in terms of conditions on the associated bound quiver $(Q, I)$. A classification of the special biserial algebras which are minimal representation-infinite has been given by Ringel in \cite{CMR1}. There is a beautiful description of all finite-dimensional indecomposable modules over special biserial algebras: they are either string modules or band modules or non-uniserial projective-injective modules, see \cite{BC}, \cite{WW}.

\begin{defn}
A bound quiver $(Q, I)$ is \textbf{special biserial} if it satisfies the following conditions:
\begin{enumerate}
  \item[(S1)] Each vertex in $Q$ is the source of at most two arrows and the target of at most two arrows.
  \item[(S2)] For an arrow $\alpha$ in $Q$ there is at most one arrow $\beta$ and at most one arrow $\gamma$  such that $\alpha\beta\not\in I$ and $\gamma\alpha\not\in I$.
\end{enumerate}

If the ideal I is generated by paths, the bound quiver $(Q, I)$ is \textbf{string}.\\
An algebra is called \textbf{special biserial} (or \textbf{string}) if it is isomorphic to $kQ/I$ with $(Q, I)$ a special biserial bound quiver (or a string bound quiver, respectively).\\
An algebra is called \textbf{string quadratic} if the ideal $I$ is generated by paths of length two.
\end{defn}

\begin{rem}\label{cases}
Note that every gentle algebra is a string algebra. In \cite{AGL15} all the nearly Frobenius structures for a gentle algebra ${\mathcal{A}}$ associated to $Q$ are determined, where $Q$ is a finite, connected
and acyclic quiver. The natural question is if we can generalize this result for the family of string algebras.
\end{rem}

The first step is to study the family of string quadratic algebras.\\
\begin{rem}\label{rem1}
If ${\mathcal{A}}$ is a string quadratic algebra that is not gentle then there is at least one vertex in one of the following local situation: 
$$1)
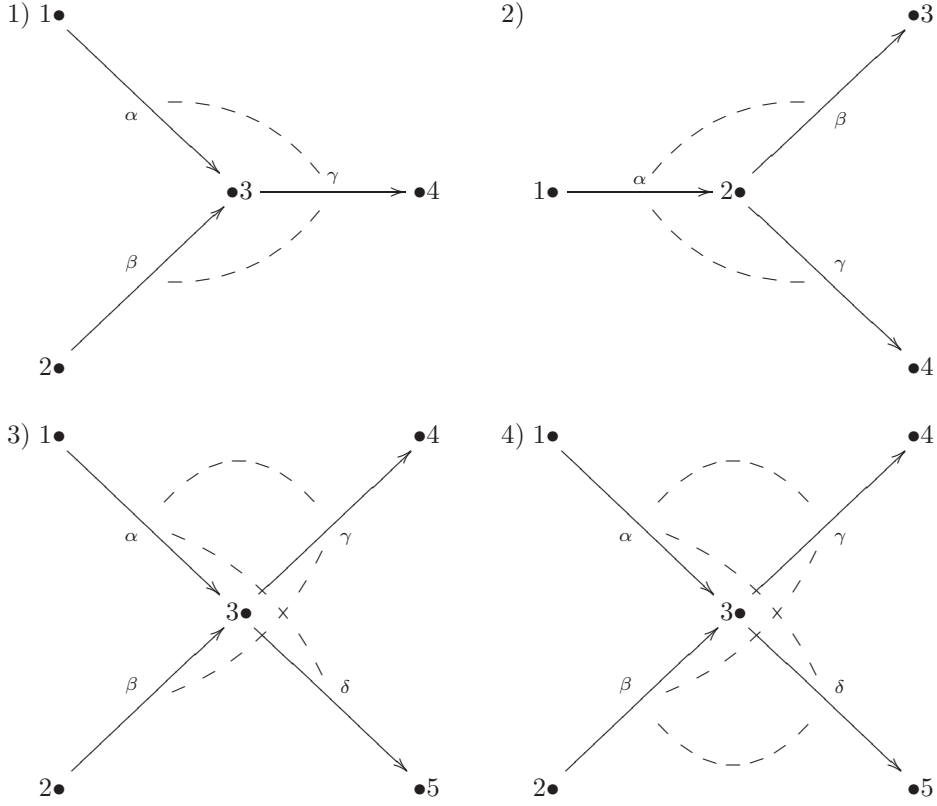
\xymatrix{
1 \bullet\ar[rrdd]_\alpha&&&&\\
&\ar@/^1pc/@{--}[rrd]&&&\\
&&\bullet 3\ar[rr]^\gamma&&\bullet 4\\
&\ar@/_1pc/@{--}[rru]&&&\\
2 \bullet\ar[rruu]^\beta&&&&
}\quad \quad
2)\xymatrix{
&&&&\bullet 3\\
&&&&\\
1 \bullet\ar[rr]^\alpha& \ar@/^1pc/@{--}[rru]\ar@/_1pc/@{--}[rrd]&2 \bullet\ar[rrdd]^\gamma\ar[rruu]_\beta&&\\
&&&&\\
&&&&\bullet 4
}$$

$$3)\xymatrix{
1 \bullet\ar[rrdd]_\alpha&&&&\bullet 4\\
&\ar@/^1pc/@{--}[rrdd]\ar@/^2pc/@{--}[rr]&&&\\
&&3 \bullet\ar[rruu]_\gamma\ar[rrdd]^\delta&&\\
&\ar@/_1pc/@{--}[rruu]&&&\\
2 \bullet\ar[rruu]^\beta&&&&\bullet 5
}\quad \quad 4) \xymatrix{
1 \bullet\ar[rrdd]_\alpha&&&&\bullet 4\\
&\ar@/^1pc/@{--}[rrdd]\ar@/^2pc/@{--}[rr]&&&\\
&&3 \bullet\ar[rruu]_\gamma\ar[rrdd]^\delta&&\\
&\ar@/_1pc/@{--}[rruu]\ar@/_2pc/@{--}[rr]&&&\\
2 \bullet\ar[rruu]^\beta&&&&\bullet 5
}$$

$\bullet$ In the first case we have that:\\
$$\begin{array}{rcl}
  \Delta(e_1) & = & a_1e_1\otimes e_1+a_2\alpha\otimes e_1 \\
  \Delta(e_2) & = & b_1e_2\otimes e_2+b_2\beta\otimes e_2 \\
  \Delta(e_3) & = & c_1e_3\otimes e_3+c_2\gamma\otimes e_3+c_3e_3\otimes\alpha+c_4\gamma\otimes\alpha+c_5e_3\otimes\beta+c_6\gamma\otimes\beta \\
  \Delta(e_4) & = & d_1e_4\otimes e_4+d_2e_4\otimes \gamma
\end{array}$$

Let us evaluate on the arrows to obtain conditions about the coefficients above. 

For the arrow $\alpha$, 
$$\Delta(\alpha)=a_1e_1\otimes\alpha+a_2\alpha\otimes\alpha=c_1\alpha\otimes e_3+c_3\alpha\otimes\alpha+c_5\alpha\otimes\beta$$
so $a_1=c_1=c_5=0$ and $a_2=c_3$. For $\beta$ and $\gamma$,
$$\Delta(\beta)=b_1e_2\otimes\beta+b_2\beta\otimes\beta=a_2\beta\otimes\alpha$$
and
$$\Delta(\gamma)=d_1\gamma\otimes e_4+d_2\gamma\otimes\gamma=c_2\gamma\otimes\gamma;$$
from the first equation we conclude that $a_2=b_1=b_2=0$ and from the second one $d_1=0$ and $d_2=c_2$.

As a consecuence we obtain that
$$\Delta(e_1)=\Delta(e_2)=0$$
$$\Delta(e_3)=a\gamma\otimes e_3+b\gamma\otimes\alpha+c\gamma\otimes\beta$$
$$\Delta(e_4)=ae_4\otimes\gamma.$$

$\bullet$ Symmetrically, in the second case:
$$\Delta(e_3)=\Delta(e_4)=0$$
$$\Delta(e_2)=ae_2\otimes\alpha+b\beta\otimes\alpha+c\gamma\otimes\alpha$$
$$\Delta(e_1)=a\alpha\otimes e_1.$$

$\bullet$ In the third case the relations are $\alpha\delta=0$, $\alpha\gamma=0$ and $\beta\gamma=0$ and
$$\begin{array}{rcl}
  \Delta(e_1) & = & a_1e_1\otimes e_1+a_2\alpha\otimes e_1 \\
  \Delta(e_2) & =   & b_1e_2\otimes e_2+b_2\beta\otimes e_2+b_3\beta\delta\otimes e_2 \\
    \Delta(e_3) & = & c_1e_3\otimes e_3+c_2e_3\otimes \alpha+c_3e_3\otimes\beta+c_4\gamma\otimes e_3+c_5\delta\otimes e_3+c_6\gamma\otimes\beta+c_7\gamma\otimes\alpha+c_8\delta\otimes\beta+c_9\delta\otimes\alpha\\
  \Delta(e_4) & = & d_1e_4\otimes e_4+d_2e_4\otimes\gamma\\
  \Delta(e_5) & = & f_1e_5\otimes e_5+f_2e_5\otimes\delta+f_3e_5\otimes \beta\delta
\end{array}
$$
As before, we will evaluate on the arrows. First consider
$$\Delta(\alpha)=a_1e_1\otimes\alpha+a_2\alpha\otimes\alpha=c_1\alpha\otimes e_3+c_2\alpha\otimes\alpha +c_3\alpha\otimes\beta,$$
then $a_1=c_1=c_3=0$ and $a_2=c_2$. Secondly,
$$\Delta(\beta)=b_1e_2\otimes\beta+b_2\beta\otimes\beta+b_3\beta\delta\otimes\beta=a_2\beta\otimes\alpha+c_5\beta\delta\otimes e_3+c_8\beta\delta\otimes\beta+c_9 \beta\delta\otimes\alpha$$
therefore $a_2=b_1=b_2=c_5=c_9=0$ and $b_3=c_8$. Finally, for the last two arrows we obtain the following equations:
$$\Delta(\gamma)=d_1\gamma\otimes e_4+d_2\gamma\otimes\gamma=c_4\gamma\otimes\gamma$$
$$\Delta(\delta)=f_1\delta\otimes e_5+f_2\delta\otimes\delta+f_3\delta\otimes \beta\delta=c_4\gamma\otimes\delta+c_6\gamma\otimes\beta\delta+b_3\delta\otimes\beta\delta;$$

from the first equation we deduce that $d_1=0$ and $d_2=c_4$ and from the second one  $f_1=f_2=c_4=c_6=0$ and $f_3=b_3$.
In conclusion,
$$\Delta(e_1)=\Delta(e_4)=0$$
$$\Delta(e_2)=b_3\beta\delta\otimes e_2$$
$$\Delta(e_3)=c_7\gamma\otimes\alpha+b_3\delta\otimes\beta$$	
$$\Delta(e_5)=b_3e_5\otimes \beta\delta.$$

$\bullet$ In the last case we have that $\alpha\delta=0$, $\beta\gamma=0$, $\alpha\gamma=0$ and $\beta\delta=0$. As before,
$$\begin{array}{rcl}
  \Delta(e_1) & = & a_1e_1\otimes e_1+a_2\alpha\otimes e_1 \\
  \Delta(e_2) & =   & b_1e_2\otimes e_2+b_2\beta\otimes e_2+b_3\beta\delta\otimes e_2 \\
    \Delta(e_3) & = & c_1e_3\otimes e_3+c_2e_3\otimes \alpha+c_3e_3\otimes\beta+c_4\gamma\otimes e_3+c_5\delta\otimes e_3+c_6\gamma\otimes\beta+c_7\gamma\otimes\alpha+c_8\delta\otimes\beta+c_9\delta\otimes\alpha\\
  \Delta(e_4) & = & d_1e_4\otimes e_4+d_2e_4\otimes\gamma\\
  \Delta(e_5) & = & f_1e_5\otimes e_5+f_2e_5\otimes\delta+f_3e_5\otimes \beta\delta
\end{array}
$$

Evaluating again on the arrows we obtain the following equations:
$$\begin{array}{rcl}
  \Delta(\alpha) & = & a_1e_1\otimes\alpha+a_2\alpha\otimes\alpha=c_1\alpha\otimes e_3+c_2\alpha\otimes \alpha+c_3\alpha\otimes\beta\\
  \Delta(\beta) & =   & b_1e_2\otimes\beta+b_2\beta\otimes\beta=a_2\beta\otimes\alpha \\
    \Delta(\gamma) & = & d_1\gamma\otimes e_4+d_2\gamma\otimes\gamma=c_4\gamma\otimes \gamma+c_5\delta\otimes \gamma\\
  \Delta(\delta) & = & f_1\delta\otimes e_5+f_2\delta\otimes\delta=d_2\gamma\otimes\delta\\
  
\end{array}
$$

and analogously to the previous cases, we conclude that
$$\Delta(e_1)=\Delta(e_2)=\Delta(e_4)=\Delta(e_5)=0.$$
$$\Delta(e_3)=c_6\gamma\otimes\beta+c_7\gamma\otimes\alpha+c_8\delta\otimes\beta+c_9\delta\otimes\alpha.$$
\end{rem}

\begin{thm}\label{stringquadratic}
If  $\mathcal{A}=\frac{\k Q}{\mathcal{I}}$ is a string quadratic algebra but not a gentle algebra, then it has a non-trivial nearly Frobenius algebra structure.
\end{thm}
\begin{proof}
We will construct a non-zero nearly Frobenius coproduct for this type of algebras. Note that there is at least one vertex $p\in Q_{0}$ in one of the four possible cases of Remark \ref{rem1}. Let us describe the coproduct in each of them.

Suppose that we are locally in the first case. Analogously to the Remark \ref{rem1} we can define a coproduct
$$\Delta:{\mathcal{A}}\rightarrow {\mathcal{A}}\otimes {\mathcal{A}}$$ such that $\Delta(e_{q})=b\gamma\otimes \alpha+c\gamma\otimes\beta$, where $b,c\in\k$ and $\alpha, \beta, \gamma$ are as in Remark  \ref{rem1} and  $\Delta(e_i)=0$ for any other vertex $i\in Q_{0}$.

To check that $\Delta$ is well defined we need to verifiy that  $\Delta(\gamma)=\Delta(\alpha)=\Delta(\beta)=0$. 

For the first arrow $\gamma$, $\Delta(\gamma)=\Delta(q)(1\otimes\gamma)=(b\gamma\otimes \alpha+c\gamma\otimes\beta)(1\otimes \gamma)=b\gamma\otimes\alpha\gamma+c\gamma\otimes\beta\gamma=0$.

In the case of $\alpha$ and $\beta$ we have that
$\Delta(\alpha)=(\alpha\otimes 1)(b\gamma\otimes \alpha+c\gamma\otimes\beta)=b\alpha\gamma\otimes \alpha+c\alpha\gamma\otimes\beta=0$ and $\Delta(\beta)=(\beta\otimes 1)(b\gamma\otimes \alpha+c\gamma\otimes\beta)=b\beta\gamma\otimes \alpha+c\beta\gamma\otimes\beta=0$.

This proves that $\Delta$ is a nearly Frobenius coproduct.

In the second case we can proceed in the same way as before and prove that $\Delta:{\mathcal{A}}\rightarrow {\mathcal{A}}\otimes {\mathcal{A}}$ defined as  $\Delta(e_{q})=b\beta\otimes\alpha+c\gamma\otimes\alpha$ and $\Delta(e_i)=0$ otherwise is a nearly Frobenius coproduct.

In the third local possibility let us define $\Delta$ as follows, $\Delta(e_{q})=c\gamma\otimes\alpha$, for $c\in\k$ and $\Delta(e_i)=0$ otherwise. It is straightforward to verifiy that $\Delta$ is well defined and a nearly Frobenius coproduct.

Finally, in the last case we define $\Delta(e_{q})=a\gamma\otimes\beta+b\gamma\otimes\alpha+c\delta\otimes\beta+d\delta\otimes\alpha$, with $a, b, c, d\in \k$ and $\Delta(e_i)=0$ for any other vertex $i\in Q_{0}$.

\end{proof}

Now, we study the general case, that is, ${\mathcal{A}}$ is a string algebra.

\begin{thm}
Let ${\mathcal{A}}$ be a string algebra, not quadratic. Then, ${\mathcal{A}}$ admits at least one non-trivial nearly Frobenius coproduct.
\end{thm}
\begin{proof}
Since ${\mathcal{A}}$ is not gentle there are extra monomial relations. If one of the extra relations is of lenght $2$ we are locally in one of the four cases of Remark \ref{rem1} and we conclude that ${\mathcal{A}}$ is nearly Frobenius.

If all the extra relations are of lenght greater that 2 choose one relation named $r=\alpha_{1}\cdots\alpha_{s}$. Locally, we are in the following situation:

$$\scalebox{0.9}
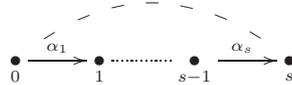
{\xymatrix{
\underset{0}{\bullet}\ar[r]^{\al_1}\ar@/^2pc/@{--}[rrr]&\underset{1}{\bullet}\ar@{..}[r]&\underset{s-1}{\bullet}\ar[r]^{\al_{s}}&\underset{s}{\bullet}
}}$$
where from every  vertex arrows might come in or out in the string scheme.\\

\emph{Claim} $\Delta{(e_{i})}=\alpha_{i+1}\cdots\alpha_{s}\otimes \alpha_{1}\cdots \alpha_{i}$ if $1\leq i\leq s-1$ and $\Delta(e_{j})=0$ otherwise is a coproduct in $A$:

If $1\leq i<j\leq s-1$,

 $\Delta(\alpha_{i+1}\cdots \alpha_j)=\Delta(e_{i})\alpha_{i+1}\cdots \alpha_j=\alpha_{i+1}\cdots \alpha_j \Delta(e_{j})=\alpha_{i+1}\cdots \alpha_{s}\otimes \alpha_{1}\cdots \alpha_{j}.$

On the other hand,

 $\Delta(\alpha_{1}\cdots \alpha_{i})= \Delta(e_{0})\alpha_{1}\cdots \alpha_{i}=\alpha_{1}\cdots \alpha_{i}\Delta(e_{i})=0$ and analogously, $\Delta(\alpha_{i+1}\cdots \alpha_{s})=0$.

Suppose there is a path $w$ from a vertex $p$ not involved in the relation to an intermediate vertex $1\leq i\leq s-1$. We can consider two different cases.

If $w$ contains $\alpha_{1}\cdots \alpha_{i} $ the reasoning  is similar to $\Delta(\alpha_{1}\cdots \alpha_{i})$. If not, there is an arrow $\beta_{j}$ such that $t(\beta_{j})=j$ is an intermediate vertex and since $A$ is string $\beta_{j}\alpha_{j+1}=0$.
Then we conclude that $\Delta(w)=0$. The case were there is a path from an intermediate vertex to another vetex not involved in $r$ is analogous and the result holds.

\end{proof}


\section{Nearly Frobenius structures on toupie algebras}
In this section we first prove that canonical algebras only admit the trivial nearly Frobenius algebra estructure. Then we characterize when a toupie algebra has a non-trivial nearly Frobenius structure.

\begin{defn}
Canonical algebras were introduced in \cite{RCM}. Let $\k$ be a field, $\mathbf{n}=\bigl(n_1,\cdots, n_t\bigr)$ be a sequence of $t\geq 2$ positive integers (\emph{weights}), and $\boldsymbol{\lambda}=\bigl(\lambda_3,\dots ,\lambda_t\bigr)$ be a sequence of pairwise distinct elements of $\k^\times$. A \textbf{canonical algebra} of type $(\mathbf{p}, \boldsymbol{\lambda})$ is an algebra $\displaystyle{\Lambda(\mathbf{p}, \boldsymbol{\lambda})= \frac{\k Q}{\mathcal{I}}}$ where $Q$ is
$$
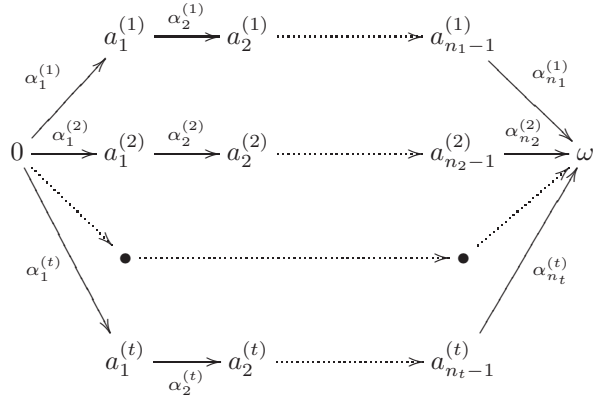
\xymatrix{
& a_1^{(1)}\ar[r]^{\al_2^{(1)}}&a_2^{(1)}\ar@{.>}[rr]&&a_{n_1-1}^{(1)}\ar[rd]^{\al_{n_1}^{(1)}}&\\
0\ar[ru]^{\al_1^{(1)}}\ar[ddr]_{\al_1^{(t)}}\ar@{.>}[rd]\ar[r]^{\al_1^{(2)}}& a_1^{(2)}\ar[r]^{\al_2^{(2)}}& a_2^{(2)}\ar@{.>}[rr]&&a_{n_2-1}^{(2)}\ar[r]^{\al_{n_2}^{(2)}} &\omega\\
&\bullet\ar@{.>}[rrr]&&&\bullet\ar@{.>}[ru]&\\
&a_1^{(t)}\ar[r]_{\al_2^{(t)}}& a_2^{(t)}\ar@{.>}[rr]&&a_{n_t-1}^{(t)}\ar[ruu]_{\al_{n_t}^{(t)}}&
}$$
$\al^{(i)}=\al_1^{(i)}\al_2^{(i)}\dots \al_{n_{i}}^{(i)}$,
and $\mathcal{I}$ is the ideal in the path algebra $\k Q$ generated by the following linear combinations of paths from $0$ to $\omega$:
$$\mathcal{I}=\langle\alpha^{(1)}-\lambda_{i}\alpha^{(2)}-\alpha^{(i)}: i=3, \cdots ,t \rangle.$$
\end{defn}

\begin{thm}\label{canonicas}
Let $\mathcal{A}$ be a canonical algebra over a field $\k$, then $\operatorname{Frobdim}\mathcal{A}=0$.
\end{thm}
\begin{proof}
Since $\mathcal{A}$ is a canonical algebra then $\displaystyle{\mathcal{A}=\frac{\k Q}{\mathcal{I}}}$ with $Q$ as in the previous figure, and $$\mathcal{I}=\langle\alpha^{(1)}-\lambda_{i}\alpha^{(2)}-\alpha^{(i)}: i=3, \cdots ,t \rangle.$$
If we have a coproduct $\Delta:\mathcal{A}\rt\mathcal{A}\ot\mathcal{A}$ the next condition is required $$\Delta(e_{i})=\Delta(e_{i})(1\otimes e_{i})=(e_{i}\otimes 1)\Delta(e_{i}), \quad\mbox{for}\, i=1,\dots, t.$$
This implies that the coproduct in $e_{0}$ is
$$\Delta(e_{0})= a_{0}e_{0}\otimes e_{0}+\sum_{j=1}^{t}\sum_{i=1}^{n_{j}} a_{i}^{j}\alpha^{(j)}_{1}\cdots \alpha^{(j)}_{i}\otimes e_{0}$$
and the coproduct in $e_{\omega}$ is
$$\Delta(e_{\omega})= b_{0}e_{\omega}\otimes e_{\omega}+\sum_{j=1}^{t}\sum_{i=1}^{n_{j}} b_{i}^{j}e_{\omega}\otimes \alpha^{(j)}_{i}\cdots \alpha^{(j)}_{n_{j}}.$$
The coproduct in $\alpha^{(1)}$ is given by
$$\Delta\bigl(\alpha^{(1)}\bigr)=\Delta(e_{0})\bigl(1\otimes \alpha^{(1)}\bigr)=\bigl(\alpha^{(1)}\otimes 1\bigr)\Delta(e_{\omega}).$$
Then
$$a_0e_0\otimes \al^{(1)}+\sum_{j=1}^{t}\sum_{i=1}^{n_{j}} a_{i}^{j}\alpha^{(j)}_{1}\cdots \alpha^{(j)}_{i}\otimes \alpha^{(1)}=b_{0}\al^{(1)}\otimes e_{\omega}+\sum_{j=1}^{t}\sum_{i=1}^{n_{j}} b_{i}^{j}\al^{(1)}\otimes \alpha^{(j)}_{i}\cdots \alpha^{(j)}_{n_{j}}$$
By comparison we deduce that $a_0=b_0=0$ and
$$\Delta\bigl(\alpha^{(1)}\bigr)=\sum_{i=1}^ta^{i}_{n_{i}}\al^{(i)}\ot\al^{(1)}=\sum_{j=1}^tb^{j}_1\al^{(1)}\otimes \al^{(j)}.$$
Finally, if we replace $\al^{(i)}$ by $\al^{(1)}-\lambda_i\al^{(2)}$ for $i\geq 3$ we get that
$$\Delta(e_0)=a\al^{(1)}\ot e_0,\quad \Delta(e_\omega)=ae_\omega\otimes \al^{(1)}.$$
As before
$$\Delta\bigl(\alpha^{(2)}\bigr)=\Delta(e_{0})\bigl(1\otimes \alpha^{(2)}\bigr)=\bigl(\alpha^{(2)}\otimes 1\bigr)\Delta(e_{\omega}).$$
Then, $\Delta\bigl(\alpha^{(2)}\bigr)=a\al^{(1)}\ot\al^{(2)}=a\al^{(2)}\ot\al^{(1)}$. Therefore $\Delta(e_0)=\Delta(e_\omega)=0$. This implies that $\Delta\equiv 0.$
\end{proof}

\begin{defn}
A quiver $Q$ is called \textbf{toupie} if it has a unique source $0$ and a unique sink $\omega$, and, for any other vertex $x$ there is exactly one arrow having
$x$ as source and exactly one arrow having $x$ as target:
$$
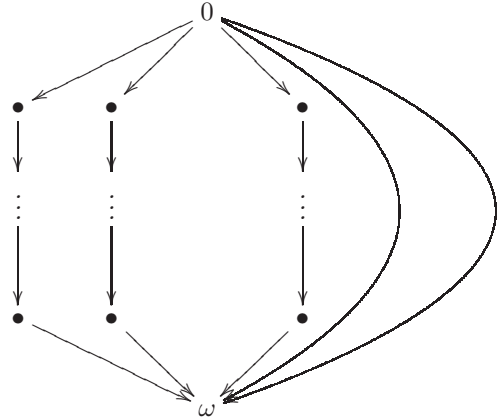
\xymatrix{
&& 0\ar[lld]\ar[ld]\ar[rd]\ar@/^6pc/[dddd]\ar@/^9pc/[dddd] &\\
\bullet\ar[d]&\bullet\ar[d] &&\bullet\ar[d]\\
\vdots\ar[d]&\vdots\ar[d] && \vdots\ar[d]\\
\bullet\ar[rrd]&\bullet \ar[rd]&& \bullet\ar[ld]\\
&&\omega &
}$$

We say that $A$ is a \textbf{toupie algebra} if $\displaystyle{\mathcal{A}=\frac{\k Q}{\mathcal{I}}}$ with $Q$ a toupie quiver, and $\mathcal{I}$ any admissible ideal.
\end{defn}

\begin{prop}
If $Q$ is a commutative diamond, then $\displaystyle{\mathcal{A}=\frac{\k Q}{\mathcal{I}}}$ has $\operatorname{Frobdim}\mathcal{A}=1$. Moreover, the only nearly Frobenius structure   over $\mathcal{A}$ is the linear structure in each branch.
\end{prop}
\begin{proof}
The quiver $Q$ associated to $\mathcal{A}$ is of the form:

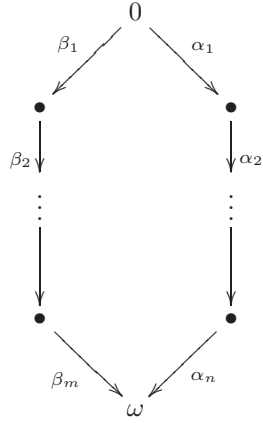
\begin{center}
$\xymatrix{
& 0\ar[ld]_{\beta_1}\ar[rd]^{\al_1} &\\
\bullet\ar[d]_{\beta_2} &&\bullet\ar[d]^{\al_2}\\
\vdots\ar[d] && \vdots\ar[d]\\
\bullet \ar[rd]_{\beta_{m}}&& \bullet\ar[ld]^{\al_n}\\
&\omega &
}$
\end{center}

and $\mathcal{I}=\langle \al_1\cdots\al_n-\beta_1\cdots\beta_m\rangle$. This means that $\al_1\cdots\al_n=\beta_1\cdots\beta_m$ in $\mathcal{A}$.  \\
If $\Delta:\mathcal{A}\rightarrow\mathcal{A}\otimes\mathcal{A}$ is a nearly Frobenius coproduct then
$$
\Delta(e_0)=(e_0\otimes 1)\Delta(e_0)=\Delta(e_0)(1\otimes e_0)=a_0e_0\ot e_0+\sum_{i=1}^na_i\al_1\dots\al_i\ot e_0+\sum_{j=1}^mb_j\beta_1\dots\beta_j\ot e_0$$
and
$$\begin{array}{ccl}
    \Delta(e_\omega)=(e_\omega\otimes 1)\Delta(e_\omega)=\Delta(e_\omega)(1\otimes e_\omega) & = &\displaystyle{ c_0e_\omega\ot e_\omega+\sum_{i=1}^nc_ie_\omega\otimes\al_i\dots\al_n} \\
     & + &\displaystyle{ \sum_{j=1}^md_je_\omega\otimes \beta_j\dots\beta_m.}
  \end{array}
$$
If we use that $\Delta(\al_1\dots\al_n)=\Delta(e_0)(1\otimes\al_1\dots\al_n)=(\al_1\dots\al_n\otimes 1)\Delta(e_\omega)$, and $\al_1\cdots\al_n=\beta_1\cdots\beta_m$ we deduce that
$$\Delta(e_0)=a\al_1\dots\al_n\otimes e_0,\quad\mbox{and}\quad\Delta(e_\omega)=ae_\omega\otimes\al_1\dots\al_n$$
where $a=a_n + b_m=c_1 + d_1.$
As a consequence of these equalities we have that $$\Delta(e_i)=a\al_i\dots\al_n\otimes\al_1\dots\al_{i-1},\;\mbox{for}\;i=1,\dots n-1$$ and
$$\Delta(e_{n+i})=a\beta_i\dots\beta_m\otimes\beta_1\dots\beta_{i-1},\;\mbox{for}\; i=1,\dots ,m-1.$$
\end{proof}

The next corollary is a generalization of the previous proposition and is proved in an analogous way. We call the algebra involved a generalized commutative diamond.

\begin{cor}\label{diamondg}
If $\displaystyle{\mathcal{A}=\frac{\k Q}{\mathcal{I}}}$ with $Q$ a quiver with $t$ branches ($t\geq 2$) and $\mathcal{I}=\langle\alpha^{(i)}-\alpha^{(1)}: i=2,\cdots t\rangle$. Then $\mathcal{A}$ has $\operatorname{Frobdim}\mathcal{A}=1$. Moreover, the only nearly Frobenius structure over $\mathcal{A}$ is the linear structure in each branch.
\end{cor}

\begin{thm}\label{toupie}
Let $\mathcal{A}$ be a toupie algebra over a field $\k$. If $m$, the number of monomial relations, is zero and $\mathcal{\mathcal{A}}$ is not the linear quiver $A_{n}$, then $\operatorname{Frobdim}\mathcal{A}=0$, except from the case of the (generalized) commutative diamond, in which $\operatorname{Frobdim}\mathcal{A}=1$.

\end{thm}
\begin{proof}
Let us first consider an order in the branches of the toupie algebra. Suppose that there are $t$ branches. We can modify the non-monomial relations to have the first $D$ branches to be linearly independent $\{\alpha^{(1)}, \cdots \alpha^{(D)}\}$ and for $i=D+1,\cdots ,t$, $\alpha^{(i)}=\sum_{j=1}^{D}\lambda_{ij}\alpha^{(j)}$. Let us call $n_{j}$ the length of $\alpha^{(j)}$
Now, as in Theorem \ref{canonicas},
$$\Delta(e_{0})= a_{0}e_{0}\otimes e_{0}+\sum_{j=1}^{t}\sum_{i=1}^{n_{j}} a_{i}^{j}\alpha^{(j)}_{1}\cdots \alpha^{(j)}_{i}\otimes e_{0}$$
and the coproduct in $e_{\omega}$ is
$$\Delta(e_{\omega})= b_{0}e_{\omega}\otimes e_{\omega}+\sum_{j=1}^{t}\sum_{i=1}^{n_{j}} b_{i}^{j}e_{\omega}\otimes \alpha^{(j)}_{i}\cdots \alpha^{(j)}_{n_{j}}.$$
The coproduct in $\alpha^{(1)}$ is given by
$$\Delta\bigl(\alpha^{(1)}\bigr)=\Delta(e_{0})\bigl(1\otimes \alpha^{(1)}\bigr)=\bigl(\alpha^{(1)}\otimes 1\bigr)\Delta(e_{\omega}).$$
Then
$$a_0e_0\otimes \al^{(1)}+\sum_{j=1}^{t}\sum_{i=1}^{n_{j}} a_{i}^{j}\alpha^{(j)}_{1}\cdots \alpha^{(j)}_{i}\otimes \alpha^{(1)}=b_{0}\al^{(1)}\otimes e_{\omega}+\sum_{j=1}^{t}\sum_{i=1}^{n_{j}} b_{i}^{j}\al^{(1)}\otimes \alpha^{(j)}_{i}\cdots \alpha^{(j)}_{n_{j}}$$
By comparison we deduce that
$$\Delta\bigl(\alpha^{(1)}\bigr)=\sum_{i=1}^ta_{n_i}^{i}\al^{(i)}\ot\al^{(1)}=\sum_{j=1}^tb_1^{j}\al^{(1)}\otimes \al^{(j)}.$$
Finally, if we replace $\al^{(i)}$ by $\sum_{j=1}^{D}\lambda_{ij}\alpha^{(j)}$ for $i> D$ we obtain that
$$\Delta(e_0)=a\al^{(1)}\ot e_0,\quad \Delta(e_\omega)=ae_\omega\otimes \al^{(1)}.$$
In the case that the toupie algebra is the generalized diamond see Corollary \ref{diamondg}. If not, it has at least two linearly independent branches $\alpha^{(1)}$ and $\alpha^{(2)}$,
$$\Delta\bigl(\alpha^{(2)}\bigr)=\Delta(e_{0})\bigl(1\otimes \alpha^{(2)}\bigr)=\bigl(\alpha^{(2)}\otimes 1\bigr)\Delta(e_{\omega}).$$
Then, $\Delta\bigl(\alpha^{(2)}\bigr)=a\al^{(1)}\ot\al^{(2)}=a\al^{(2)}\ot\al^{(1)}$. Therefore $\Delta(e_0)=\Delta(e_\omega)=0$. This implies that $\Delta\bigl(\alpha^{(i)}\bigr)=0$ and $\Delta\equiv 0.$
\end{proof}

Next we will consider the case of toupie algebras with monomial relations.
\begin{prop}\label{m}
Consider $\mathcal{A}=\frac{A_{n}}{\mathcal{I}}$ with $\mathcal{I}\neq 0$ a monomial ideal. Then $\operatorname{Frobdim}\mathcal{A}\geq 1$.
\end{prop}
\begin{proof}
Suppose that $\mathcal{A}$ has a relation of length $r+1$ of the form:

$$\scalebox{0.8}
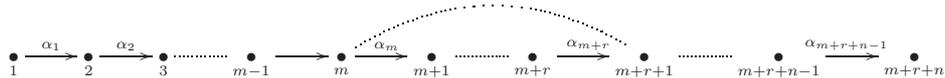
{\xymatrix{
\underset{1}{\bullet}\ar[r]^{\al_1}&\underset{2}{\bullet}\ar[r]^{\al_2}&\underset{3}{\bullet}\ar@{..}[r]& \underset{m-1}{\bullet}\ar[r]&\underset{m}{\bullet}\ar[r]^{\al_m}\ar@/^2pc/@{..}[rrr]&\underset{m+1}{\bullet}\ar@{..}[r]&\underset{m+r}{\bullet}\ar[r]^{\al_{m+r}}&\underset{m +r+1}{\bullet}\ar@{..}[r]&\underset{m+r+n-1}{\bullet}\ar[r]^{\al_{m+r+n-1}}&\underset{m+r+n}{\bullet}
}}$$

Let us now define a coproduct in $\mathcal{A}$:\\
$\Delta(e_{m+i})=\al_{m+i}\dots \al_{m+r}\ot \al_{m}\dots \al_{m+i-1}$ if $i=1,\dots r$ and $\Delta(e_{p})=0 $ otherwise.

If $p\leq m$ and there exists a non-null path $w$ from $p$ to $m+i$ with $1\leq i\leq r$ then,
$\Delta(w)=\Delta(e_{p})w=0$ by definition and on the other hand
$\Delta(w)=w\Delta(e_{m+i})=0$
since $\al_{m}\dots \al_{m+i-1}$ must be included in $w$.

If $p\geq m+r+1$ the argument is analogous.

Finally, if $1\leq p<q \leq r$, the only path from $m+p$ to $m+q$ is $\al_{m+p}\dots \al_{m+q-1}$ and,
$\Delta(\al_{m+p}\dots \al_{m+q-1})=\Delta(e_{m+p})\al_{m+p}\dots \al_{m+q-1}= \al_{m+p}\dots \al_{m+q-1}\Delta(e_{m+q})$ so we conclude that $\Delta$ is a coproduct.

\end{proof}

The next theorem describes toupie algebras with nearly Frobenius structures.
\begin{thm}
Let $\mathcal{A}$ be a toupie algebra over a field $\k$ and $m$ the number of branches with monomial relations, then
\begin{enumerate}
  \item[(1)] if $m=0$
  \begin{enumerate}
    \item[(a)] and $\mathcal{A}$ is the linear quiver $A_{n}$ or the (generalized) commutative diamond we have that $\operatorname{Frobdim}\mathcal{A}=1$,
    \item[(b)] in other case $\operatorname{Frobdim}\mathcal{A}=0$,
  \end{enumerate}
  \item[(2)] if $m>0$ then $\displaystyle{\operatorname{Frobdim}\mathcal{A}\geq 1}.$

\end{enumerate}
\end{thm}
\begin{proof}
\begin{enumerate}
  \item[\emph{(1)}]
  \begin{enumerate}
    \item[\emph{(a)}] This result is a consequence of Theorem 1 of \cite{AGL15} and the Corollary \ref{diamondg}.
    \item[\emph{(b)}] It is the Theorem \ref{toupie}.
  \end{enumerate}
  \item[\emph{(2)}] If there is only one branch and has monomial relations is the case of Proposition \ref{m}. If not, using Theorem 4 of \cite{AGL15}, the coproduct over the branches is zero except for the monomial branches, moreover, the coproduct on the first and the last arrow of the monomial branches is zero. Then, combining this result with Proposition \ref{m} over any monomial branch we have that
          $$\operatorname{Frobdim}\mathcal{A}\geq 1.$$
 \end{enumerate}

\end{proof}


\section{Final Comment}

After computing the nearly Frobenius structures in some representative classes of algebras, we observe that some local situations guarantee the existence of non-trivial nearly Frobenius structures. We summarized them in the following result.

\begin{thm}

Let ${\mathcal{A}} = \frac{\Bbbk Q}{I}$ be a finite dimensional algebra. If $Q$ has a local situation in a vertex $v \in Q_0$ as follows

$$1) 
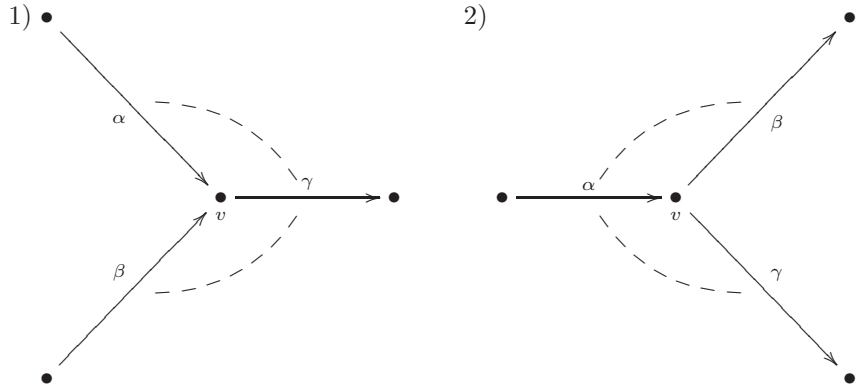
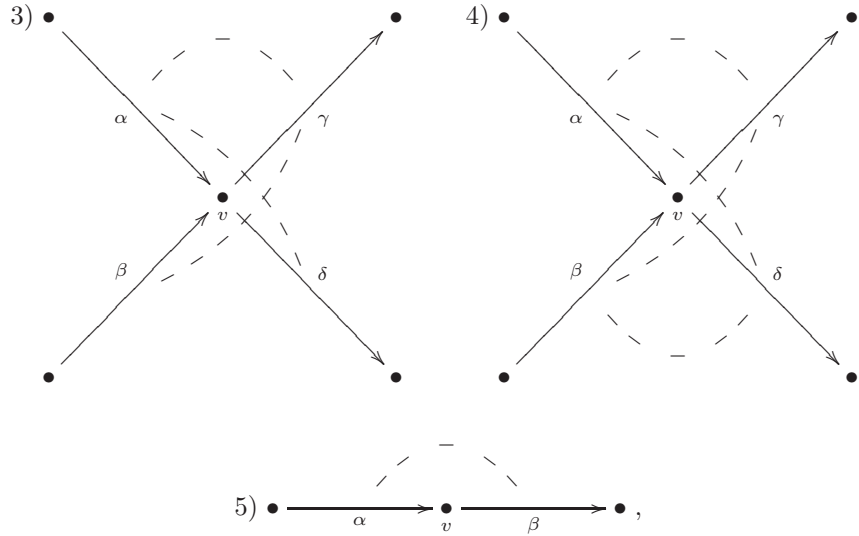
\xymatrix{
 \bullet\ar[rrdd]_\alpha&&&&\\
&\ar@/^1pc/@{--}[rrd]&&&\\
&& \underset{v}\bullet \ar[rr]^\gamma&&\bullet \\
&\ar@/_1pc/@{--}[rru]&&&\\
 \bullet\ar[rruu]^\beta&&&&
}\quad \quad 2) \xymatrix{
&&&&\bullet \\
&&&&\\  \bullet\ar[rr]^\alpha &  \ar@/^1pc/@{--}[rru]\ar@/_1pc/@{--}[rrd]& \underset{v} \bullet\ar[rrdd]^\gamma\ar[rruu]_\beta&&\\
&&&&\\
&&&&\bullet
}$$

$$3) \xymatrix{
 \bullet\ar[rrdd]_\alpha&&&&\bullet \\
&\ar@/^1pc/@{--}[rrdd]\ar@/^2pc/@{--}[rr]&&&\\
&& \underset{v} \bullet\ar[rruu]_\gamma\ar[rrdd]^\delta&&\\
&\ar@/_1pc/@{--}[rruu]&&&\\
 \bullet\ar[rruu]^\beta&&&&\bullet
}\quad \quad 4) \xymatrix{
  \bullet\ar[rrdd]_\alpha&&&&\bullet \\
&\ar@/^1pc/@{--}[rrdd]\ar@/^2pc/@{--}[rr]&&&\\
&& \underset{v} \bullet\ar[rruu]_\gamma\ar[rrdd]^\delta&&\\
&\ar@/_1pc/@{--}[rruu]\ar@/_2pc/@{--}[rr]&&&\\
 \bullet\ar[rruu]^\beta&&&&\bullet
}$$
\vspace{1cm}
$${5) \xymatrix{{\bullet}\ar[rr]_{\alpha}&\ar@/^2pc/@{--}[rr]&\underset{v}{\bullet} \ar[rr]_{\beta}& &{\bullet}
}},$$
then ${\mathcal{A}}$ has a non-trivial structure of nearly Frobenius algebra.
\end{thm}
\begin{proof}

The coproduct in the first four cases is analogous to the ones in Theorem \ref{stringquadratic}. For the last case the coproduct is the following

$$
\Delta(x) = \left\{ \begin{array}{cc}
 \beta\otimes \alpha & \mbox{ if } x = e_v \\
 0 & \mbox{otherwise}
 \end{array}\right.$$

\end{proof}





\end{document}